\documentclass[11pt]{article}

\usepackage[margin=1in]{geometry}
\usepackage{amsmath,amssymb,amsthm,mathtools}
\usepackage{thmtools}
\usepackage{mathrsfs}
\usepackage{enumitem}
\usepackage{array}
\usepackage{booktabs}
\usepackage{longtable}
\usepackage{multirow}
\usepackage{graphicx}
\usepackage{tikz}
\usepackage[nameinlink,noabbrev]{cleveref}
\usepackage{verbatim}
\usepackage{xcolor}

\newtheorem{theorem}{Theorem}[section]
\newtheorem{lemma}[theorem]{Lemma}
\newtheorem{proposition}[theorem]{Proposition}
\newtheorem{corollary}[theorem]{Corollary}

\newtheorem{question}[theorem]{Question}

\theoremstyle{definition}
\newtheorem{definition}[theorem]{Definition}

\newtheorem{remark}[theorem]{Remark}

\newcommand{\Aut}{\operatorname{Aut}}

\newcommand{\dist}{\operatorname{dist}}
\newcommand{\id}{1'}

\newcommand{\pijh}{p_{ij}^h}

\title{\textbf{Relation Algebra Representations from Distance-Regular Graphs}}
\author{Eli Atkins \\ Southern Illinois University}
\date{\today}

\begin{document}

\maketitle

\begin{abstract}
    We describe a general method for constructing representations of finite integral symmetric relation algebras from distance-regular graphs. Given a distance-regular graph of diameter $d$, the distances between vertices induces a coloring of the complete graph with $d$ colors, and we show that this coloring yields a representation of finite integral symmetric relation algebra on $d+1$ atoms. We then introduce a necessary and sufficient condition for when such a representation is algebraic, proving that this occurs if and only if the distance-regular graph is also distance-transitive.
    
    We study the diameter-3 case of this method in detail, and we express a condition for the representation's mandatory cycles in terms of the distance-regular graph's intersection array. We apply this result to give a positive answer to an open question of Roger Maddux; namely, whether the relation algebra $30_{65}$ has a representation on a finite set. The representation is given on 42 points, and arises from the second subconstituent of the Hoffman-Singleton graph. We further use this method to describe an infinite class of finite representations of $26_{65}$ and the smallest possible representation of $31_{65}$.
\end{abstract}

\section{Preliminary Material}\label{sec:preliminaries}

\subsection{Relation algebras and representations}

In the combinatorial representation theory of relation algebras, the main questions asked of a finite integral representable relation algebra $A$ are as follows:

\begin{enumerate}
    \item What is the spectrum of $A$, i.e., what is 
    \[
    \{ n \leq\omega : A \text{ has a square representation on a set of cardinality } n \}
    \]
    \item What are the ``nice'' representations of $A$?  ``Nice" representations include finite group representations (especially cyclic group representations), algebraic representations, and other highly symmetric representations. 
    \item Can the representations of $A$ be completely characterized? 
\end{enumerate}

For example, in \cite{AndMadd}, the spectrum of each finite three-atom relation algebra was determined. In \cite{Alm24}, the spectrum of cyclic group representations was determined for all seven finite integral symmetric relation algebras on three atoms.  In a work not yet submitted \cite{AlmAt}, the algebraic representations for those same seven algebras are mostly characterized. 

In this paper, we introduce two new types of ``nice'' representations, namely distance-regular representations and distance-transitive representations.  Such representations correspond to distance-regular  and distance-transitive graphs, respectively. Background on distance-regular and distance-transitive graphs may be found in the comprehensive survey paper of van Dam, Koolen, and Tanaka \cite{DistanceSurvey}, as well as the monograph by Brouwer, Cohen, and Neumaier \cite{DistanceBook}. We will prove that every distance-regular (distance-transitive) graph induces a representation of a finite integral relation algebra, and as a corollary give new representations of eight finite relation algebras. One of these, namely relation algebra $30_{65}$, is the first known finite representation.

This use of regular graph structures arising from association schemes and coherent configurations in constructing relation algebra representations has appeared previously in remarks of Steve Comer  (see \cite{Comer1} and \cite{Comer2}) including a simple example arising from a distance-regular graph, namely the cubical graph on 8 vertices. However, the general distance-regular framework and its consequences do not appear to have been more thoroughly developed in the literature.

The 2008 paper \cite{Alm08} contains the explanation of the graph-coloring/representation equivalence. In this paper, we consider a representation to be a coloring of a complete graph by the relations $R_i$, where the mandatory and forbidden triangle configurations agree with the composition table of the relation algebra. 

One may refer to \cite{Maddux} for the numbering system and the composition tables for some small relation algebras. 

There are $65$ finite symmetric integral relation algebras on 4 atoms. Each algebra has symmetric atoms $1'$, $a$, $b$, and $c$; and the distinct diversity cycles (up to equivalence) are
\[
aaa,\quad bbb,\quad ccc,\quad abb,\quad baa,\quad acc,\quad caa,\quad bcc,\quad cbb,\quad abc.
\]

We include the algebras mentioned in this paper, along with their mandatory diversity cycles, in Table \ref{tab:cycles}.

\begin{table}[hbt]
    \centering
    \begin{tabular}{c|cccccccccc}
          RA  & $aaa$ & $bbb$ & $ccc$ & $abb$ & $baa$ & $acc$ & $caa$ & $bcc$ & $cbb$ & $abc$  \\
          \hline
          $26_{65}$  & $aaa$ &  &  & $abb$ &  &  &  &  &  & $abc$  \\
          $27_{65}$  & $aaa$ & $bbb$ &  & $abb$ & $baa$ &  &  &  &  & $abc$  \\
          $28_{65}$  & $aaa$ &  &  & $abb$ &  & $acc$ &  &  &  & $abc$  \\
          $30_{65}$  & $aaa$ &  & $ccc$ & $abb$ & $baa$ &  & $caa$ &  &  & $abc$  \\
          $31_{65}$  & $aaa$ & $bbb$ & $ccc$ & $abb$ & $baa$ &  & $caa$ &  &  & $abc$  \\
          $57_{65}$  & $aaa$ & $bbb$ &  & $abb$ & $baa$ & $acc$ & $caa$ & $bcc$ &  & $abc$  \\
          $59_{65}$  & $aaa$ &  & $ccc$ & $abb$ & $baa$ & $acc$ & $caa$ & $bcc$ &  & $abc$  \\
          $61_{65}$  & $aaa$ & $bbb$ & $ccc$ & $abb$ & $baa$ & $acc$ & $caa$ & $bcc$ &  & $abc$  \\
    \end{tabular}
    \caption{Mandatory diversity cycles for selected algebras}
    \label{tab:cycles} 
\end{table}

In this setting, satisfying the mandatory and forbidden triangle properties is equivalent to the following condition: for all $h,i,j$ and all $\{x,y\}\in R_h$, the truth of the statement
\[
\exists \ z \in X \text{ such that } \{x,z\}\in R_i \text{ and } \{z,y\}\in R_j
\]
depends only on $h,i,j$, and not on the choice of $\{x,y\}\in R_h$.

Equivalently, for each triple $(h,i,j)$, either
\[
\{z\in X : \{x,z\}\in R_i \text{ and } \{z,y\}\in R_j\} = \emptyset
\]
for all $\{x,y\}\in R_h$, or it is nonempty for all such $\{x,y\}$. One can see how these conditions correspond to a composition table for any pair of relations.

A stronger condition is that for each $h,i,j$ the number
\[
\pijh = \left|\{z\in X : \{x,z\}\in R_i \text{ and } \{z,y\}\in R_j\}\right|
\]
is independent of the choice of $\{x,y\}\in R_h$. That is, the composition depends only on when this constant is nonzero. This stronger uniformity occurs in the distance-regular graph construction considered in this paper.

Throughout this paper, the represented algebras will always be finite, symmetric, and integral. We will also refer to the relations $R_i$ as color classes, as the graph representations are given by graph colorings.

\subsection{Algebraic representations}

\begin{definition}
Consider a representation over a finite set of a finite symmetric integral relation algebra with relations $R_i$. The representation is said to be \emph{algebraic} if each color class $R_i$ is an orbital of the automorphism group of the colored complete graph determined by the representation.
\end{definition}

Equivalently, the relations defined on edges, or color classes, coincide with the orbits of the automorphism group on ordered pairs of vertices.

\subsection{Distance-regular graphs}

\begin{definition}
A connected graph $\Gamma$ with diameter $d$ is \emph{distance-regular} if for all integers $h,i,j$ and all vertices $x,y$ with $\dist(x,y)=h$, the number
\[
p_{ij}^h = \left|\{z \in V(\Gamma): \dist(x,z)=i,\ \dist(y,z)=j\}\right| \tag{$\ast$}
\]
depends only on $h,i,j$ and not on the particular pair $(x,y)$. The numbers $\pijh$ are called the \emph{intersection numbers} of $\Gamma$.
\end{definition}

For a distance-regular graph of diameter $d$, one usually writes
\[
b_i = p_{1,i+1}^i,\qquad
c_i = p_{1,i-1}^i,\qquad
a_i = p_{1,i}^i
\]
for $0\le i\le d$, with the conventions
\[
c_0=0,\qquad b_d=0,\qquad a_i = b_0 - b_i - c_i.
\]
The array
\[
\{b_0,b_1,\dots,b_{d-1};c_1,c_2,\dots,c_d\}
\]
is the \emph{intersection array} of $\Gamma$.

Note that in a distance-regular graph, the full set of intersection numbers $\pijh$ is determined by the intersection array. For the diameter-3 case, we will derive and list these formulas later.

\subsection{Distance-transitive graphs}

\begin{definition}
A connected graph $\Gamma$ is \emph{distance-transitive} if whenever $x,y,u,v\in V(\Gamma)$ satisfy
\[
\dist(x,y)=\dist(u,v),
\]
there exists an automorphism $g\in \Aut(\Gamma)$ such that $g(x)=u$ and $g(y)=v$.
\end{definition}

\begin{remark}
Distance-transitivity implies distance-regularity. The converse is false in general.
\end{remark}

\section{The construction from distance-regular graphs}\label{sec:construction}

In this section we prove that every distance-regular graph gives rise to a finite symmetric integral relation algebra representation.

\begin{definition}
    Let $\Gamma$ be a finite connected distance-regular graph of diameter $d$, and let $X=V(\Gamma)$, the vertex set of $\Gamma$. For each $i\in \{0,1,\dots,d\}$ define the relation $R_i$ to be pairs of vertices at distance $i$ in $\Gamma$. Thus $R_0$ is the identity relation, and for $i\ge 1$, the relation $R_i$ corresponds to the set of ordered pairs of vertices at graph distance $i$.
\end{definition}

\begin{theorem}\label{thm:mainconstruction}
Let $\Gamma$ be a finite connected distance-regular graph of diameter $d$. Then the relations
\[
R_0,R_1,\dots,R_d
\]
yield a representation of a finite symmetric integral relation algebra with exactly $d+1$ atoms.
\end{theorem}

\begin{proof}
Let $X=V(\Gamma)$. Since $\Gamma$ is connected with diameter $d$, every ordered pair $(x,y)\in X\times X$ has a unique distance
\[
\dist(x,y)\in \{0,1,\dots,d\}.
\]
Hence the relations $R_0,R_1,\dots,R_d$ partition the edge set $E(\Gamma)$.

By definition, $R_0$ consists exactly of those pairs $(x,y)$ with $\dist(x,y)=0$, and this holds if and only if $x=y$. Thus
\[
R_0=\{(x,x):x\in X\},
\]
so $R_0$ is the identity relation on $X$.

Each $R_i$ is symmetric, since graph distance is symmetric:
\[
\dist(x,y)=\dist(y,x).
\]

Now fix $h,i,j\in \{0,1,\dots,d\}$ and choose any $(x,y)\in R_h$. Then
\[
\left|\{z\in X:(x,z)\in R_i \text{ and } (z,y)\in R_j\}\right|
=
\left|\{z\in X:\dist(x,z)=i,\ \dist(y,z)=j\}\right|.
\]
Because $\Gamma$ is distance-regular, this number depends only on $i,j,h$, not on the particular choice of $(x,y)\in R_h$. Denote it by $\pijh$.

Therefore relational composition is determined by these constants. Thus the relations $R_0,R_1,\dots,R_d$ form the atoms of a finite symmetric integral relation algebra representation on $X$. There are exactly $d+1$ such relations, so the represented algebra has exactly $d+1$ atoms.
\end{proof}

\medskip

The theorem may be interpreted as follows: each distance class in the graph becomes one color, and the intersection numbers of the graph determine the composition table of the resulting relation algebra.

\begin{corollary}
Every finite connected distance-regular graph of diameter $d$ induces a coloring of the complete graph on its vertex set with $d$ colors, and this coloring is a representation of a finite symmetric integral relation algebra.
\end{corollary}

\begin{proof}
This is just a reformulation of \Cref{thm:mainconstruction}.
\end{proof}

\section{When is the representation algebraic?}\label{sec:algebraic}

We now determine exactly when the representation arising from \Cref{thm:mainconstruction} is algebraic.

\begin{theorem}\label{thm:algebraiciff}
Let $\Gamma$ be a finite connected distance-regular graph, and let
\[
R_0,R_1,\dots,R_d
\]
be the relations defined by graph distance. Then the resulting representation is algebraic if and only if $\Gamma$ is distance-transitive.
\end{theorem}

\begin{proof}
Let $X=V(\Gamma)$, and let $\mathcal{G}$ denote the complete graph on $X$ whose edges are colored according to graph distance in $\Gamma$.

Suppose first that the representation is algebraic. Then each relation $R_i$ is an orbital of the automorphism group of the colored graph $\mathcal{G}$. In particular, for any two ordered pairs $(x,y)$ and $(u,v)$ with
\[
\dist(x,y)=\dist(u,v)=i,
\]
there exists a color-preserving automorphism $g$ of $\mathcal{G}$ such that
\[
g(x)=u,\qquad g(y)=v.
\]
Since color $1$ in $\mathcal{G}$ is exactly adjacency in $\Gamma$, every color-preserving automorphism of $\mathcal{G}$ is an automorphism of $\Gamma$. Hence for any two ordered pairs at the same graph distance, there is an automorphism of $\Gamma$ carrying one to the other. Thus $\Gamma$ is a distance-transitive graph.

Conversely, suppose that $\Gamma$ is distance-transitive. Let $(x,y)$ and $(u,v)$ be ordered pairs satisfying
\[
\dist(x,y)=\dist(u,v)=i.
\]
By distance-transitivity, there exists $g\in \Aut(\Gamma)$ such that
\[
g(x)=u,\qquad g(y)=v.
\]
Every automorphism of $\Gamma$ preserves graph distance, so $g$ preserves every color class in the distance coloring of the complete graph. Therefore each relation $R_i$ is a single orbital of the color-preserving automorphism group. Hence the representation is algebraic.
\end{proof}

\medskip

Note that the diameter-$2$ case of \Cref{thm:algebraiciff} is the strongly regular graph situation: algebraic two-color representations arise precisely from those strongly regular graphs that are distance-transitive, which is equivalently the rank-$3$ strongly-regular graphs. More information on this case can be found in \cite{AlmAt}.

\section{The diameter-3 case}\label{sec:diameter3}

From this point onward, let $\Gamma$ be a distance-regular graph of diameter $3$. The induced representation has an identity atom and diversity atoms
\[
a,\ b,\ c,
\]
where, in this paper, the atoms correspond to distances as follows:
\[
b \leftrightarrow \text{distance }1,\qquad
a \leftrightarrow \text{distance }2,\qquad
c \leftrightarrow \text{distance }3.
\]

\subsection{Setup and notation}

Let the intersection array of $\Gamma$ be
\[
\{b_0,b_1,b_2;c_1,c_2,c_3\}.
\]
As was mentioned previously,
\[
a_i = b_0 - b_i - c_i \qquad (0\le i\le 3),
\]
with $b_3=0$ and $c_0=0$.

For a fixed pair $(x,y)$ with $\dist(x,y)=h$, define
\[
N_{ij}(h)
=
\left|\{z\in V(\Gamma): \dist(x,z)=i,\ \dist(y,z)=j\}\right|.
\]
By distance-regularity,
\[
N_{ij}(h)=p_{ij}^h,
\]
and this quantity depends only on $h,i,j$; see \cite[pp.~15]{DistanceSurvey}.

Following Maddux's composition tables, we use the labels
\[
aaa,\quad bbb,\quad ccc,\quad abb,\quad baa,\quad acc,\quad caa,\quad bcc,\quad cbb,\quad abc.
\]
for the possible diversity cycle types. Since we are concerned only with the existence or nonexistence of these triangle configurations in the graph coloring, and not with counts relative to a chosen base edge, symmetry implies that the ordering of the diversity cycles is irrelevant. Hence there are only these ten possible cycle types.

\subsection{General principle}

\begin{proposition}\label{prop:trianglepositivity}
Let $x,y$ be vertices with $\dist(x,y)=h$. A triangle with $\{x,y\}$ as its base, and its other edges formed by a point $z$ such that $\dist(x,z)=i$ and $\dist(y,z)=j$ occurs if and only if
\[
p_{ij}^h > 0.
\]
\end{proposition}

\begin{proof}
By definition, a vertex $z$ forms such a triangle with the base pair $(x,y)$ exactly when
\[
\dist(x,z)=i,\qquad \dist(y,z)=j.
\]
The number of such vertices is precisely
\[
\left|\{z\in V(\Gamma): \dist(x,z)=i,\ \dist(y,z)=j\}\right| = N_{ij}(h)=p_{ij}^h.
\]
Thus such a triangle exists if and only if $p_{ij}^h>0$.
\end{proof}

\medskip

Thus the problem of determining which triangles occur is reduced to determining which intersection numbers are positive.

\begin{lemma}\label{lem:khpijk}
Let $\Gamma$ be a distance-regular graph of diameter $d$, and for $0\le r\le d$ let
\[
\Gamma_r(x):=\{y\in V(\Gamma):\dist(x,y)=r\},
\qquad
k_r:=|\Gamma_r(x)|.
\]
Then $k_r$ is independent of the choice of $x$. Moreover, for all $h,i,j$,
\begin{equation}\label{kp}
    k_h\,p_{ij}^h = k_i\,p_{hj}^i.
\end{equation}
\end{lemma}

\begin{proof}
The independence of $k_r$ from the choice of $x$ is standard for distance-regular graphs; see \cite[p.~8]{DistanceSurvey}.

Towards proving (\ref{kp}), fix $x\in V(\Gamma)$, and consider the set
\[
S=\{(y,z)\in V(\Gamma)\times V(\Gamma): \dist(x,y)=h,\ \dist(x,z)=i,\ \dist(y,z)=j\}.
\]
We count $|S|$ in two ways.

First, if $y\in \Gamma_h(x)$, then the number of vertices $z$ such that
\[
\dist(x,z)=i,\qquad \dist(y,z)=j
\]
is, by the definition of the intersection numbers, equal to $p_{ij}^h$. Since $|\Gamma_h(x)|=k_h$, this gives
\[
|S|=k_h\,p_{ij}^h.
\]

Second, if $z\in \Gamma_i(x)$, then the number of vertices $y$ such that
\[
\dist(x,y)=h,\qquad \dist(y,z)=j
\]
is $p_{hj}^i$. Since $|\Gamma_i(x)|=k_i$, this gives
\[
|S|=k_i\,p_{hj}^i.
\]

Equating the two expressions for $|S|$ yields
\[
k_h\,p_{ij}^h = k_i\,p_{hj}^i.
\]
\end{proof}

This equality is simple but useful.

\begin{corollary}\label{cor:layer-sizes}
Let $\Gamma$ be a distance-regular graph of diameter $3$. Then
\[
k_1 = b_0,\qquad
k_2 = \frac{b_0 b_1}{c_2},\qquad
k_3 = \frac{b_0 b_1 b_2}{c_2 c_3}.
\]
\end{corollary}

The proof of this corollary may be found in \cite[pp.~11]{DistanceSurvey}.

The proofs of the following three propositions involve only standard manipulations of intersection number identities for distance-regular graphs. For readability, these derivations are found in Appendix A.

\subsubsection*{Base color $b$ (distance 1)}

For a base pair at distance $1$, the relevant counts are
\[
p_{11}^1,\ p_{12}^1,\ p_{13}^1,\ p_{22}^1,\ p_{23}^1,\ p_{33}^1.
\]

\begin{proposition}\label{prop:baseb}
For a base pair at distance $1$, the possible triangle types and the corresponding intersection numbers are:
\begin{align*}
bbb &: p_{11}^1 = a_1,\\
abb &: p_{12}^1 = b_1,\\
baa &: p_{22}^1 = \frac{b_1a_2}{c_2},\\
bcc &: p_{33}^1 = \frac{b_1b_2a_3}{c_2c_3},\\
cbb &: p_{13}^1 = 0,\\
abc &: p_{23}^1 = \frac{b_1b_2}{c_2}.
\end{align*}
In particular, a listed triangle type occurs if and only if the corresponding intersection number is nonzero.
\end{proposition}

\subsubsection*{Base color $a$ (distance 2)}

For a base pair at distance $2$, the relevant counts are
\[
p_{11}^2,\ p_{12}^2,\ p_{13}^2,\ p_{22}^2,\ p_{23}^2,\ p_{33}^2.
\]

\begin{proposition}\label{prop:basea}
For a base pair at distance $2$, the possible triangle types and the corresponding intersection numbers are:
\begin{align*}
aaa &: p_{22}^2 = \frac{b_1c_2+a_2^2+b_2c_3-b_0-a_1a_2}{c_2},\\
abb &: p_{11}^2 = c_2,\\
baa &: p_{12}^2 = a_2,\\
acc &: p_{33}^2 = \frac{b_0b_1b_2}{c_2c_3}-b_2-\frac{b_2(a_2+a_3-a_1)}{c_2},\\
caa &: p_{23}^2 = \frac{b_2(a_2+a_3-a_1)}{c_2},\\
abc &: p_{13}^2 = b_2.
\end{align*}
In particular, a listed triangle type occurs if and only if the corresponding intersection number is nonzero.
\end{proposition}

\subsubsection*{Base color $c$ (distance 3)}

For a base pair at distance $3$, the relevant counts are
\[
p_{11}^3,\ p_{12}^3,\ p_{13}^3,\ p_{22}^3,\ p_{23}^3,\ p_{33}^3.
\]

\begin{proposition}\label{prop:basec}
For a base pair at distance $3$, the possible triangle types and the corresponding intersection numbers are:
\begin{align*}
ccc &: p_{33}^3 = \frac{b_0b_1b_2}{c_2c_3}-1-a_3-\frac{a_3(a_3-a_1)+b_2c_3-b_0}{c_2},\\
acc &: p_{23}^3 = \frac{a_3(a_3-a_1)+b_2c_3-b_0}{c_2},\\
caa &: p_{22}^3 = \frac{c_3(a_2+a_3-a_1)}{c_2},\\
bcc &: p_{13}^3 = a_3,\\
cbb &: p_{11}^3 = 0,\\
abc &: p_{12}^3 = c_3.
\end{align*}
In particular, a listed triangle type occurs if and only if the corresponding intersection number is nonzero.
\end{proposition}

\subsection{A summary table for the diameter-3 case}\label{table}

For ease of reference, we summarize the preceding formulas in a single table. Note that in many cases, due to symmetry, there is more than one choice for the positivity criterion. As a result, we list here the simplest such choice for each diversity cycle.

\begin{center}
\renewcommand{\arraystretch}{1.25}
\begin{tabular}{ccc}
\toprule
Triangle type & Relevant intersection number & Simplest Positivity criterion \\
\midrule
$aaa$ & $p_{22}^2$ & $\frac{b_1c_2+a_2a_2+b_2c_3-b_0-a_1a_2}{c_2}>0$ \\
$bbb$ & $p_{11}^1$ & $a_1>0$ \\
$ccc$ & $p_{33}^3$ & $\frac{b_0b_1b_2}{c_2c_3}-1-a_3-\frac{a_3(a_3-a_1)+b_2c_3-b_0}{c_2}>0$ \\
$abb$ & $p_{12}^1$ or $p_{11}^2$ & $b_1>0$ \\
$baa$ & $p_{12}^2$ or $p_{22}^1$ & $a_2>0$ \\
$acc$ & $p_{23}^3$ or $p_{33}^2$ & $\frac{a_3(a_3-a_1)+b_2c_3-b_0}{c_2}>0$ \\
$caa$ & $p_{22}^3$ or $p_{23}^2$ & $\frac{c_3(a_2+a_3-a_1)}{c_2}>0$ \\
$bcc$ & $p_{13}^3$ or $p_{33}^1$ & $a_3>0$ \\
$cbb$ & $p_{13}^1$ or $p_{11}^3$ & $0>0 \text{ (never positive)}$ \\
$abc$ & $p_{13}^2$ or $p_{23}^1$ or $p_{12}^3$ & $b_2>0$ \\
\bottomrule
\end{tabular}
\end{center}

The fact that $p_{13}^1$ and $p_{11}^3$ are always $0$ gives our first indication that not all representable finite symmetric integral relation algebras on 4 atoms can be represented according to this method.

\begin{theorem}\label{thm:diam3determinesRA}
Let $\Gamma$ be a distance-regular graph of diameter $3$. Then the cycle table of the induced four-atom relation algebra is determined entirely by the intersection array of $\Gamma$.
\end{theorem}

\begin{proof}
By \Cref{prop:trianglepositivity}, a diversity triangle type occurs if and only if the relevant intersection number is positive. By the preceding propositions, in diameter $3$ each relevant intersection number can be expressed in terms of the intersection-array parameters. Hence the full set of allowed triangles is determined by the intersection array, and therefore so is the induced relation algebra.
\end{proof}

Thus, once the intersection array is known, one can determine the represented algebra without analyzing the graph itself.

\section{The Hoffman--Singleton example and the algebra $30_{65}$}\label{sec:HS}

\subsection{The graph}

Let $\Lambda$ denote the second subconstituent of the Hoffman--Singleton graph, as described in \cite{DistanceBook}. Then $\Lambda$ is a distance-regular graph of diameter $3$ with intersection array
\[
\{6,5,1;1,1,6\}.
\]

We define a coloring of the complete graph on $V(\Lambda)$, as described previously, by choosing the diversity atom for an edge between distinct vertices $x$ and $y$ as follows:
\[
\{x,y\}=
\begin{cases}
b, & \dist(x,y)=1,\\
a, & \dist(x,y)=2,\\
c, & \dist(x,y)=3.
\end{cases}
\]
By \Cref{thm:mainconstruction}, this yields a representation of a finite symmetric integral relation algebra with diversity atoms $a$, $b$, and $c$.

Note that since $\Lambda$ is distance-transitive, \Cref{thm:algebraiciff} implies that the resulting representation is algebraic.

\subsection{Computing the cycle table}

Applying the formulas from \Cref{sec:diameter3} to the intersection array of $\Lambda$, we obtain the following conclusions about the diversity cycles:
\begin{itemize}
    \item \textbf{Mandatory triangle types: [aaa ccc abb baa caa abc]}
    \item \textbf{Forbidden triangle types: [bbb acc bcc cbb]}
\end{itemize}

Equivalently, the composition table can be described by:
\begin{align*}
a;a &= \id + a + b + c,\\
a;b &= a + b + c,\\
a;c &= a + b,\\
b;b &= \id + a,\\
b;c &= a,\\
c;c &= \id + c.
\end{align*}

\begin{theorem}\label{thm:3065}
The relation algebra induced by the second subconstituent of the Hoffman--Singleton graph is $30_{65}$.
\end{theorem}

\begin{proof}
By the preceding calculations, the composition table obtained from $\Lambda$ agrees exactly with the allowed triangle types of $30_{65}$. Therefore the induced representation is a representation of $30_{65}$.
\end{proof}

\begin{corollary}
$30_{65}$ has a finite representation.
\end{corollary}

To the best of our knowledge, this is the first known finite representation of $30_{65}$.

Note that this representation can be described in terms of the graph coloring. The edges colored by $c$ partition the graph into 7 disjoint monochromatic cliques of size 6. Between any pair of such cliques, there exists a perfect matching of edges colored by $b$. All other edges are colored by $a$. One can find that every representation of this algebra must be able to be described in this same way: with disjoint, equal-sized cliques in one color that are connected via perfect matchings with another color.

In unpublished notes, Maddux proves that for a representation of $30_{65}$ to exist with clique size $3$, it must be countably infinite. That is, no finite representation exists with a clique size of $3$. Maddux gives the infinite representation with clique size $3$, but it remains open whether any representation may exist with clique size $4$ or $5$.

\section{Additional examples}\label{sec:examples}

In this section we list several further examples of distance-regular graphs and the relation algebras obtained from them.

The diameter-$3$ formulas make it easy to compute the cycle structure of the represented algebra directly from the intersection array, and the same idea applies across other diameters. The following table records examples from other distance-regular graphs, where the examples were intentionally chosen so that the represented algebras are different. The graph theoretic data in this table can be found in \cite{DistanceBook}, unless a different source is noted.

\begin{center}
\small
\renewcommand{\arraystretch}{1.2}
\begin{longtable}{p{3.2cm}p{3.5cm}p{1.2cm}p{1cm}p{1.8cm}p{3cm}}
\toprule
Graph & Intersection array & Induced RA & $|V|$ & Distance-transitive? & Notes \\
\midrule
\endfirsthead

\toprule
Graph & Intersection array & Induced RA & $|V|$ & Distance-transitive? & Notes \\
\midrule
\endhead

Crown graph & $\{4,3,1;1,3,4\}$ & $26_{65}$ & $10$ & Yes \\
Icosahedral graph & $\{5,2,1;1,2,5\}$ & $27_{65}$ & $12$ & Yes & Planar graph \\
Heawood Graph & $\{3,2,2;1,1,3\}$ & $28_{65}$ & $14$ & Yes \\
Petersen Line Graph Qt39 & $\{4,2,1;1,1,4\}$ & $31_{65}$ & $15$ & No & Smallest rep of $31_{65}$ \\
(3,3)-Hamming Graph & $\{6,4,2;1,2,3\}$ & $61_{65}$ & $27$ & Yes & See Remark~\ref{Hamming} \\
Sylvester graph & $\{5,4,2;1,1,4\}$ & $59_{65}$ & $36$ & Yes & See Remark~\ref{Sylvester} \\
2nd subconstituent Hoffman--Singleton & $\{6,5,1;1,1,6\}$ & $30_{65}$ & $42$ & Yes & Only known finite rep \\
Moscow-Soicher graph & $\{110,81,12;1,18,90\}$ & $57_{65}$ & $672$ & No & Introduced in \cite{Soicher} \\
\bottomrule
\end{longtable}
\end{center}

\begin{remark}\label{Hamming}
    The representation of $61_{65}$ given in the table by the $(3,3)$-Hamming Graph on 27 points is smaller than the representation that Maddux gives in his unpublished notes on $84$ vertices, which arises from Johnson schemes.
\end{remark}

\begin{remark}\label{Sylvester}
    In \cite{Alm52-59}, Alm and Maddux show that a finite representation of $59_{65}$ exists on $113$ points. This representation derived from the Sylvester graph is smaller, on $36$ points.
\end{remark}

\subsection{The crown graphs and the algebra $26_{65}$}

Note the first example of the crown graph with 10 vertices. Crown graphs exist on $2n$ vertices for all $n \geq 3$. One can show that all of these crown graphs are distance-transitive and produce algebraic representations of $26_{65}$.

\begin{theorem}\label{thm:crown}
    For every integer $n \geq 3$, the crown graph $K_{n,n} \setminus M$, where $M$ is a perfect matching, is a distance-transitive graph of diameter $3$ with intersection array
    \[
    \{ n-1,n-2,1;1,n-2,n-1 \}.
    \]
    Using this graph, the relation algebra representation on $2n$ vertices induced by the distance-regular graph construction described in Theorem~\ref{thm:diam3determinesRA} is $26_{65}$. Moreover, since crown graphs are distance-transitive, the resulting representation is algebraic.
\end{theorem}

\begin{proof}
    The construction of the crown graph, as well as its intersection array and the property that it is distance-transitive, can be found in \cite{DistanceBook}. Thus, we will show that this intersection array always generates the diversity cycle table for $26_{65}$.
    Given that the intersection array is $\{n-1,n-2,1;1,n-2,n-1\}$, one finds that
    \[
    a_1 = a_2 = a_3 = 0,
    \]
    since
    \[
    (n-1) - (n-2) - 1=0; \qquad (n-1)-1-(n-2)=0; \qquad (n-1)-(n-1)=0.
    \]
    Using the positivity criterion from Section~\ref{table}, one finds that the only occurring diversity cycles are
    \[
    aaa, \qquad abb, \qquad abc.
    \]
    Since the remaining diversity cycles do not occur, this gives us the cycle table for $26_{65}$. Since the crown graph is distance-transitive, Theorem~\ref{thm:algebraiciff} implies that the representation is algebraic.
\end{proof}

\subsection{The Petersen line graph and the algebra $31_{65}$}

We continue by analyzing another graph in the table. The Petersen line graph example above yields a representation of $31_{65}$ on $15$ vertices. We now show that this representation is minimal.

\begin{proposition}
In every representation of $31_{65}$, the $c$-edges form a disjoint union of monochromatic cliques.
\end{proposition}

\begin{proof}
In $31_{65}$, the diversity cycles $acc$ and $bcc$ are forbidden. Thus, if two distinct points $y$ and $z$ are both connected to a point $x$ by $c$-edges, then the edge $\{y,z\}$ cannot be colored $a$, since this would form an $acc$-triangle. Similarly, the edge $\{y,z\}$ cannot be colored $b$, since this would form a $bcc$-triangle. Therefore $\{y,z\}$ must be colored $c$.

It follows that whenever two $c$-edges share a vertex, their third edge is also colored $c$. Hence the connected components of the graph formed by the $c$-edges are monochromatic $c$-cliques.
\end{proof}

\begin{theorem}
Every representation of $31_{65}$ has at least $15$ vertices.
\end{theorem}

\begin{proof}
Let $X$ be the vertex set of a graph representation of $31_{65}$. By the previous proposition, the $c$-edges partition $X$ into disjoint monochromatic $c$-cliques.

Since $ccc$ is a mandatory cycle, every $c$-edge must be contained in a $ccc$-triangle. Therefore each $c$-clique has size at least $3$.

We now show that there must be at least five such $c$-cliques. Let $V$ and $W$ be two distinct $c$-cliques, and choose $v\in V$. Since $V$ and $W$ are distinct $c$-cliques, every edge from $v$ to a vertex of $W$ is colored either $a$ or $b$.

For each $w\in W$, there must be a point $x$ such that the edges $\{v,x\}$ and $\{w,x\}$ are colored b. Indeed, if $\{v,w\}$ is colored $a$, this follows from the mandatory cycle $abb$, and if $\{v,w\}$ is colored $b$, this follows from the mandatory cycle $bbb$. Thus for each $w\in W$, there exists a point $x_w\in X$ such that
\[
\{v,x_w\} \text{ is colored } b
\qquad\text{and}\qquad
\{w,x_w\} \text{ is colored } b.
\]

The point $x_w$ cannot lie in $V$, since then $\{v,x_w\}$ would be a $c$-edge. Similarly, $x_w$ cannot lie in $W$, since then $\{w,x_w\}$ would be a $c$-edge. Hence each $x_w$ lies in a $c$-clique different from both $V$ and $W$.

Moreover, if $w_1,w_2\in W$ are distinct, then $x_{w_1}$ and $x_{w_2}$ cannot lie in the same $c$-clique. If they did, then $\{x_{w_1},x_{w_2}\}$ would be a $c$-edge. Since both $x_{w_1}$ and $x_{w_2}$ are $b$-adjacent to $v$, this would form a forbidden $cbb$-triangle.

Therefore the vertices $x_w$, for $w\in W$, lie in $|W|$ distinct $c$-cliques, none of which is $V$ or $W$. Since $|W|\geq 3$, there are at least
\[
2+|W|\geq 5
\]
distinct $c$-cliques.

Since each $c$-clique has size at least $3$, the total number of vertices is at least
\[
5\cdot 3=15.
\]
Thus every representation of $31_{65}$ has at least $15$ vertices.
\end{proof}

\begin{corollary}
The representation of $31_{65}$ arising from the line graph of the Petersen graph is minimal.
\end{corollary}

\begin{proof}
The line graph of the Petersen graph has $15$ vertices and, by a calculation from its intersection array, yields a representation of $31_{65}$. By the theorem, no representation of $31_{65}$ can have fewer than $15$ vertices. Therefore, this representation is minimal.
\end{proof}

\section{Open questions}\label{sec:questions}

We conclude with several questions suggested by this work.

\begin{question}
Which finite symmetric integral relation algebras can be realized by the distance-regular graph construction of \Cref{thm:mainconstruction}?
\end{question}

This asks for a structural characterization of the relation algebras that arise from distance partitions of distance-regular graphs. Can any relation algebras on 4 atoms not described in the above table be represented using this method?

\begin{question}
    Is the representation of $61_{65}$ given in the table on $27$ vertices the smallest representation of this algebra?
\end{question}

\begin{question}
    Is the representation of $59_{65}$ given in the table on $36$ vertices the smallest representation of this algebra?
\end{question}

\begin{question}
What clique sizes are possible in finite representations of $30_{65}$? In particular, can one obtain such a representation with clique size $4$ or $5$?
\end{question}

As mentioned previously, the representation constructed here has clique size $6$. Since Maddux has ruled out cliques of size $3$ in the finite case, it would be interesting to understand whether clique sizes of $4$ or $5$ are possible.

\begin{question}
Is the representation of $30_{65}$ arising from the second subconstituent of the Hoffman--Singleton graph the minimal algebraic representation with respect to the number of vertices?
\end{question}

\appendix

\section{Derivations of the Diversity Cycle Intersection Numbers}

In this appendix, we provide the calculations used in the proofs of Propositions \ref{prop:baseb}--\ref{prop:basec}. Each identity follows from standard relations among the intersection numbers of a distance-regular graph and algebraic manipulation.

\begin{proof}[Proof of Proposition \ref{prop:baseb}]
Let $x,y$ be vertices with $\dist(x,y)=1$, and for $r\in\{0,1,2,3\}$ let
\[
k_r := |\Gamma_r(x)|.
\]
By \Cref{cor:layer-sizes}, we have
\[
k_1 = b_0,\qquad
k_2 = \frac{b_0 b_1}{c_2},\qquad
k_3 = \frac{b_0 b_1 b_2}{c_2 c_3}.
\]

We also use the identity from \Cref{lem:khpijk}
\[
k_h\,p_{ij}^h = k_i\,p_{hj}^i,
\]
valid for all $h,i,j$. In particular, with $h=1$, this gives
\[
k_1\,p_{ij}^1 = k_i\,p_{1j}^i.
\]

We now compute each relevant intersection number.

First,
\[
p_{11}^1=a_1
\]
by the definition of $a_1=p_{1,1}^1$. Hence the triangle type $bbb$ occurs exactly $a_1$ times over a base pair of color $b$.

Next,
\[
p_{12}^1=b_1
\]
by the definition of $b_1=p_{1,2}^1$. Hence the triangle type $abb$ occurs exactly $b_1$ times.

Also,
\[
p_{13}^1=0.
\]
If $z$ satisfied $\dist(x,z)=1$ and $\dist(y,z)=3$, then since $\dist(x,y)=1$ we would have
\[
\dist(y,z)\le \dist(y,x)+\dist(x,z)=1+1=2,
\]
a contradiction. Thus the triangle type $cbb$ does not occur.

Now apply the identity $k_1p_{ij}^1=k_ip_{1j}^i$.

For $p_{22}^1$,
\[
k_1p_{22}^1 = k_2p_{12}^2.
\]
But $p_{12}^2=a_2$, so
\[
p_{22}^1=\frac{k_2a_2}{k_1}
=\frac{\left(\frac{b_0b_1}{c_2}\right)a_2}{b_0}
=\frac{b_1a_2}{c_2}.
\]
Thus the triangle type $baa$ occurs exactly $\frac{b_1a_2}{c_2}$ times.

For $p_{23}^1$,
\[
k_1p_{23}^1 = k_2p_{13}^2.
\]
But $p_{13}^2=b_2$, so
\[
p_{23}^1=\frac{k_2b_2}{k_1}
=\frac{\left(\frac{b_0b_1}{c_2}\right)b_2}{b_0}
=\frac{b_1b_2}{c_2}.
\]
Thus the triangle type $abc$ occurs exactly $\frac{b_1b_2}{c_2}$ times.

Finally, for $p_{33}^1$,
\[
k_1p_{33}^1 = k_3p_{13}^3.
\]
But $p_{13}^3=a_3$, so
\[
p_{33}^1=\frac{k_3a_3}{k_1}
=\frac{\left(\frac{b_0b_1b_2}{c_2c_3}\right)a_3}{b_0}
=\frac{b_1b_2a_3}{c_2c_3}.
\]
Thus the triangle type $bcc$ occurs exactly $\frac{b_1b_2a_3}{c_2c_3}$ times.

This proves all of the stated formulas. The final sentence follows immediately: a given triangle type occurs if and only if its corresponding count is positive.
\end{proof}

\begin{proof}[Proof of Proposition \ref{prop:basea}]
Let $A_0,A_1,A_2,A_3$ denote the distance matrices of $\Gamma$, so that
\[
(A_r)_{xy}=
\begin{cases}
1,&\text{if }\dist(x,y)=r,\\
0,&\text{otherwise.}
\end{cases}
\]
Because $\Gamma$ is distance-regular, these matrices satisfy
\[
A_iA_j=\sum_{h=0}^3 p_{ij}^hA_h;
\]
see equation $(5)$ in \cite[p.~15]{DistanceSurvey}. In particular, the coefficient of $A_h$ in $A_iA_j$ is $p_{ij}^h$.

The three identities
\[
p_{11}^2=c_2,\qquad p_{12}^2=a_2,\qquad p_{13}^2=b_2
\]
are exactly the equalities
\[
p_{1,1}^2=c_2,\qquad p_{1,2}^2=a_2,\qquad p_{1,3}^2=b_2
\]
for the intersection numbers $c_2,a_2,b_2$.

To compute $p_{22}^2$ and $p_{22}^3$, we use the products
\[
A_1A_1=b_0A_0+a_1A_1+c_2A_2,
\]
\[
A_1A_2=b_1A_1+a_2A_2+c_3A_3,
\]
\[
A_1A_3=b_2A_2+a_3A_3.
\]
These follow from
\[
A_iA_j=\sum_{h=0}^3p_{ij}^hA_h
\]
together with the definitions
\[
p_{1,1}^0=b_0,\quad p_{1,1}^1=a_1,\quad p_{1,1}^2=c_2,\quad p_{1,1}^3=0,
\]
\[
p_{1,2}^1=b_1,\quad p_{1,2}^2=a_2,\quad p_{1,2}^3=c_3,
\]
\[
p_{1,3}^2=b_2,\quad p_{1,3}^3=a_3.
\]

Also, using the same identity,
\[
A_2A_2=p_{22}^0A_0+p_{22}^1A_1+p_{22}^2A_2+p_{22}^3A_3.
\]

Now because matrix multiplication is associative:
\[
(A_1A_1)A_2=A_1(A_1A_2).
\]

We expand the left-hand side first. Substituting the above formula for $A_1A_1$, we get
\[
(A_1A_1)A_2=(b_0A_0+a_1A_1+c_2A_2)A_2.
\]
Since $A_0$ is the identity matrix, $A_0A_2=A_2$. Therefore
\[
(A_1A_1)A_2=b_0A_2+a_1(A_1A_2)+c_2(A_2A_2).
\]
Substituting the formulas for $A_1A_2$ and $A_2A_2$, this becomes
\[
(A_1A_1)A_2
=
b_0A_2+a_1(b_1A_1+a_2A_2+c_3A_3)
+c_2(p_{22}^0A_0+p_{22}^1A_1+p_{22}^2A_2+p_{22}^3A_3).
\]
Hence the coefficient of $A_2$ on the left-hand side is
\[
b_0+a_1a_2+c_2p_{22}^2,
\]
and the coefficient of $A_3$ on the left-hand side is
\[
a_1c_3+c_2p_{22}^3.
\]

Next expand the right-hand side:
\[
A_1(A_1A_2)=A_1(b_1A_1+a_2A_2+c_3A_3)
=b_1(A_1A_1)+a_2(A_1A_2)+c_3(A_1A_3).
\]
Substituting the three formulas above gives
\[
A_1(A_1A_2)
=
b_1(b_0A_0+a_1A_1+c_2A_2)
+a_2(b_1A_1+a_2A_2+c_3A_3)
+c_3(b_2A_2+a_3A_3).
\]
Therefore the coefficient of $A_2$ on the right-hand side is
\[
b_1c_2+a_2^2+b_2c_3,
\]
and the coefficient of $A_3$ on the right-hand side is
\[
a_2c_3+a_3c_3.
\]

Because the distance matrices are linearly independent, the coefficients of each $A_r$ must agree. Thus, comparing the coefficients of $A_2$, we obtain
\[
b_0+a_1a_2+c_2p_{22}^2=b_1c_2+a_2^2+b_2c_3,
\]
which implies that
\[
p_{22}^2=\frac{b_1c_2+a_2^2+b_2c_3-b_0-a_1a_2}{c_2}.
\]

Similarly, comparing the coefficients of $A_3$, we obtain
\[
a_1c_3+c_2p_{22}^3=a_2c_3+a_3c_3,
\]
and hence
\[
p_{22}^3=\frac{c_3(a_2+a_3-a_1)}{c_2}.
\]

For the next equality, we use Lemma~\ref{lem:khpijk}. Taking $h=3$, $i=2$, and $j=2$ gives
\[
k_3p_{22}^3=k_2p_{32}^2.
\]
Since the intersection numbers satisfy $p_{32}^2=p_{23}^2$, so
\[
k_3p_{22}^3=k_2p_{23}^2.
\]
Therefore
\[
p_{23}^2=\frac{k_3}{k_2}p_{22}^3.
\]
By Corollary~\ref{cor:layer-sizes},
\[
\frac{k_3}{k_2}
=
\frac{ \frac{b_0b_1b_2}{c_2c_3} }{ \frac{b_0b_1}{c_2} }
=
\frac{b_2}{c_3}.
\]
Substituting this and the above formula for $p_{22}^3$ leads to
\[
p_{23}^2
=
\frac{b_2}{c_3}\cdot \frac{c_3(a_2+a_3-a_1)}{c_2}
=
\frac{b_2(a_2+a_3-a_1)}{c_2}.
\]

Finally, we compute $p_{33}^2$ from a row sum. For a fixed pair of vertices $(x,y)$ with $\dist(x,y)=2$, the number of vertices at distance $3$ from $x$ is $k_3$. Such a vertex $z$ can have distance $1$, $2$, or $3$ from $y$, but not distance $0$, because $y$ itself is at distance $2$ from $x$, not $3$. Hence
\[
p_{31}^2+p_{32}^2+p_{33}^2=k_3.
\]
Now $p_{31}^2=p_{13}^2=b_2$, and $p_{32}^2=p_{23}^2$. Therefore
\[
p_{33}^2=k_3-b_2-p_{23}^2.
\]
Substituting Corollary~\ref{cor:layer-sizes} and the formula for $p_{23}^2$, we get
\[
p_{33}^2
=
\frac{b_0b_1b_2}{c_2c_3}
-b_2
-\frac{b_2(a_2+a_3-a_1)}{c_2}.
\]

This proves all of the equalities. By Proposition~\ref{prop:trianglepositivity}, each listed triangle type occurs if and only if the corresponding count is positive.
\end{proof}

\begin{proof}[Proof of Proposition \ref{prop:basec}]
The identities
\[
p_{11}^3=0,\qquad p_{12}^3=c_3,\qquad p_{13}^3=a_3
\]
follow immediately from the definitions of $b_3$, $c_3$, and $a_3$, since
\[
p_{11}^3=b_3=0,\qquad p_{12}^3=c_3,\qquad p_{13}^3=a_3.
\]

Next, by the computation in the proof of \Cref{prop:basea}, we already know that
\[
p_{22}^3=\frac{c_3(a_2+a_3-a_1)}{c_2}.
\]

To compute $p_{23}^3$, we use the identity from \Cref{lem:khpijk} with $h=2$, $i=3$, and $j=2$:
\[
k_2p_{32}^2=k_3p_{22}^3.
\]
Since $p_{32}^2=p_{23}^2$, this gives
\[
k_2p_{23}^2=k_3p_{22}^3.
\]
From \Cref{prop:basea},
\[
p_{23}^2=\frac{b_2(a_2+a_3-a_1)}{c_2},
\]
so
\[
k_2\frac{b_2(a_2+a_3-a_1)}{c_2}=k_3p_{22}^3.
\]
Using \Cref{cor:layer-sizes}, we have
\[
\frac{k_2}{k_3}=\frac{c_3}{b_2},
\]
hence
\[
p_{22}^3
=
\frac{k_2}{k_3}\cdot \frac{b_2(a_2+a_3-a_1)}{c_2}
=
\frac{c_3(a_2+a_3-a_1)}{c_2}.
\]

We now compute $p_{23}^3$ using the row sum
\[
p_{21}^3+p_{22}^3+p_{23}^3=k_2.
\]
Since
\[
p_{21}^3=p_{12}^3=c_3,
\]
it follows that
\[
p_{23}^3=k_2-c_3-p_{22}^3.
\]
Using \Cref{cor:layer-sizes},
\[
k_2=\frac{b_0b_1}{c_2},
\]
so
\[
p_{23}^3
=
\frac{b_0b_1}{c_2}-c_3-\frac{c_3(a_2+a_3-a_1)}{c_2}.
\]
Since $a_2=b_0-b_2-c_2$, this simplifies to
\[
p_{23}^3=\frac{a_3(a_3-a_1)+b_2c_3-b_0}{c_2}.
\]

Finally, to compute $p_{33}^3$, note that for a base pair at distance $3$,
\[
p_{30}^3+p_{31}^3+p_{32}^3+p_{33}^3=k_3.
\]
Here
\[
p_{30}^3=1,
\]
since the second base vertex itself is the unique vertex at distance $0$ from one endpoint and distance $3$ from the other. Also,
\[
p_{31}^3=p_{13}^3=a_3,\qquad p_{32}^3=p_{23}^3.
\]
Therefore
\[
p_{33}^3=k_3-1-a_3-p_{23}^3.
\]
Using \Cref{cor:layer-sizes}, this becomes
\[
p_{33}^3
=
\frac{b_0b_1b_2}{c_2c_3}-1-a_3-\frac{a_3(a_3-a_1)+b_2c_3-b_0}{c_2}.
\]

This proves all of the stated formulas. By \Cref{prop:trianglepositivity}, each corresponding triangle type occurs if and only if the associated count is positive.
\end{proof}

\section{Acknowledgements}

The author would like to thank his advisor, Jeremy Alm, for introducing him to the fascinating world of relation algebras and the study of their representations, as well as for his input and ideas while working on this paper.

\section{Declarations}

The author has no competing interests to declare that are relevant to the content of the article.

All relevant data are included in this article.


\begin{thebibliography}{10}

\bibitem{AlmAt}
Jeremy~F. Alm and Eli Atkins.
\newblock Algebraic representations of small relation algebras.
\newblock {\em In progress}.

\bibitem{Alm24}
Jeremy~F. Alm, Ashlee Bostic, Claire Chenault, Kenyon Coleman, and Chesney Culver.
\newblock Cyclic group spectra for some small relation algebras.
\newblock In {\em Relational and algebraic methods in computer science}, volume 14787 of {\em Lecture Notes in Comput. Sci.}, pages 19--27. Springer, Cham, [2024] \copyright 2024.

\bibitem{Alm52-59}
Jeremy~F. Alm and Roger~D. Maddux.
\newblock Finite representations for two small relation algebras.
\newblock {\em Algebra Universalis}, 79(4):Paper No. 87, 4, 2018.

\bibitem{Alm08}
Jeremy~F. Alm, Roger~D. Maddux, and Jacob Manske.
\newblock Chromatic graphs, {R}amsey numbers and the flexible atom conjecture.
\newblock {\em Electron. J. Combin.}, 15(1):Research paper 49, 8, 2008.

\bibitem{AndMadd}
Hajnal Andr\'eka and Roger~D. Maddux.
\newblock Representations for small relation algebras.
\newblock {\em Notre Dame J. Formal Logic}, 35(4):550--562, 1994.

\bibitem{DistanceBook}
A.~E. Brouwer, A.~M. Cohen, and A.~Neumaier.
\newblock {\em Distance-regular graphs}, volume~18 of {\em Ergebnisse der Mathematik und ihrer Grenzgebiete (3) [Results in Mathematics and Related Areas (3)]}.
\newblock Springer-Verlag, Berlin, 1989.

\bibitem{Comer1}
Stephen~D. Comer.
\newblock A new foundation for the theory of relations.
\newblock {\em Notre Dame J. Formal Logic}, 24(2):181--187, 1983.

\bibitem{Comer2}
Stephen~D. Comer.
\newblock Combinatorial aspects of relations.
\newblock {\em Algebra Universalis}, 18(1):77--94, 1984.

\bibitem{Maddux}
Roger~D. Maddux.
\newblock {\em Relation algebras}, volume 150 of {\em Studies in Logic and the Foundations of Mathematics}.
\newblock Elsevier B. V., Amsterdam, 2006.

\bibitem{Soicher}
Leonard~H. Soicher.
\newblock Yet another distance-regular graph related to a {G}olay code.
\newblock {\em Electron. J. Combin.}, 2:Note 1, approx. 4, 1995.

\bibitem{DistanceSurvey}
Edwin~R. van Dam, Jack~H. Koolen, and Hajime Tanaka.
\newblock Distance-regular graphs.
\newblock {\em Electron. J. Combin.}, DS22:156, 2016.

\end{thebibliography}

\end{document}